\documentclass[a4paper]{amsart}
\usepackage{amsmath,amstext,amsthm,amscd,amsopn,verbatim,amssymb, amsfonts}
\usepackage[bbgreekl]{mathbbol}

\usepackage{fullpage}

\usepackage{tikz}
\usetikzlibrary{matrix}
\usetikzlibrary{shapes}
\usetikzlibrary{arrows}
\usetikzlibrary{calc,3d}
\usetikzlibrary{decorations,decorations.pathmorphing}
\usetikzlibrary{through}
\usetikzlibrary{shapes.multipart}
\tikzset{ext/.style={circle, draw,inner sep=1pt},
int/.style={circle,draw,fill,inner sep=0, minimum size=5},
xit/.style={cross out,draw,inner sep=0, minimum size=5},
nil/.style={inner sep=1pt},
descr/.style={inner sep=1.5pt, outer sep=0.1,font=\scriptsize}}
\tikzset{xst/.style={draw, cross out, minimum size=5, }}
\tikzset{exte/.style={circle, draw,inner sep=3pt},inte/.style={circle,draw,fill,inner sep=3pt}}
\tikzset{diagram/.style={matrix of math nodes, row sep=3em, column sep=2.5em, text height=1.5ex, text depth=0.25ex}}
\tikzset{diagram2/.style={matrix of math nodes, row sep=0.5em, column sep=0.5em, text height=1.5ex, text depth=0.25ex}}

\theoremstyle{plain}
  \newtheorem{thm}{Theorem}
  \newtheorem{defi}{Definition}
  \newtheorem{prop}{Proposition}
  
  \newtheorem{lemma}{Lemma}
\theoremstyle{definition}
  \newtheorem{ex}{Example}
  \newtheorem{rem}{Remark}

\newcommand{\alg}[1]{\mathfrak{{#1}}}
\newcommand{\co}[2]{\left[{#1},{#2}\right]} 

\newcommand{\ad}{{\text{ad}}}

\newcommand{\p}{\partial}

\newcommand{\R}{{\mathbb{R}}}

\newcommand{\KGra}{{\mathsf{KGra}}}

\newcommand{\Duflo}{{\mathit{Duflo}}}
\newcommand{\assoc}{{\mathit{assoc}}}

\newcommand{\Kontsevich}{ {\mathit{Kontsevich} }}
\newcommand{\Shoikhet}{{ \mathit{Shoikhet} }}

\newcommand{\Lie}{\mathsf{Lie}}

\newcommand{\Ass}{\mathit{Assoc}}

\newcommand{\PBW}{\mathit{PBW}}
\newcommand{\CFFR}{\mathit{CFFR}}
\newcommand{\brane}{\mathit{brane}}

\newcommand{\GC}{\mathsf{GC}}

\newcommand{\bpm}{\begin{pmatrix}}
\newcommand{\epm}{\end{pmatrix}}
\newcommand{\Tpoly}{T_{\rm poly}}

\newcommand{\Dpoly}{D_{\rm poly}}
\newcommand{\ExpGC}{ {\mathrm{Exp}GC_{2,cl}^0} }
\newcommand{\curv}{{\mathit{curv}}}
\newcommand{\chain}{{\mathit{chain}}}

\newcommand{\grt}{{\mathfrak{grt}}}

\newcommand{\mU}{\mathcal{U}}
\newcommand{\mV}{\mathcal{V}}
\newcommand{\mW}{\mathcal{W}}

\DeclareMathOperator{\tr}{tr}

\begin{document}
\title{Characteristic classes in deformation quantization}
\author{Thomas Willwacher}
\address{Institute of Mathematics\\ University of Zurich\\ Winterthurerstrasse 190 \\ 8057 Zurich, Switzerland}
\email{thomas.willwacher@math.uzh.ch}


\begin{abstract}
In deformation quantization one can associate five characteristic functions to (stable) formality morphisms on cochains and chains and to ``two-brane'' formality morphisms. We show that these characteristic functions agree.
\end{abstract}
\maketitle

\section{Introduction}
Let $\Tpoly(\R^n)=\Gamma(\R^n, \wedge^\bullet T\R^n)$ be the space of multivector fields on $\R^n$ and let $\Dpoly(\R^n)$ be the space of multidifferential operators on $\R^n$. The central result of deformation quantization is M. Kontsevich's Formality Theorem \cite{K1}, stating that there is a $\Lie_\infty$ quasi-isomorphism
\[
\mU^{\Kontsevich}\colon \Tpoly(\R^n)[1] \to \Dpoly(\R^n)[1].
\] 
Here we understand $\Tpoly(\R^n)[1]$ as a Lie algebra endowed with the Schouten-Nijenhuis bracket and $\Dpoly(\R^n)[1]$ as a Lie algebra endowed with the Gerstenhaber bracket.
The differential forms $\Omega_\bullet(\R^n)$ on $\R^n$, with non-positive grading, form a Lie module over $\Tpoly(\R^n)[1]$, and similarly the (topological) Hochschild chains $C_\bullet(\R^n)=C_\bullet(C^\infty(\R^n), C^\infty(\R^n))$ form a module over the multidifferential operators $\Dpoly(\R^n)$. For a more detailed description of these objects and the actions we refer the reader to \cite{tsygan}.
It was conjectured by B. Tsygan \cite{tsygan} and shown by B. Shoikhet \cite{shoikhet} that there is a $\Lie_\infty$ quasi-isomorphism of modules
\[
\mV^{\Shoikhet} \colon C_\bullet(\R^n) \to \Omega_\bullet(\R^n).
\]
A globalized version of this statement was shown by V. Dolgushev \cite{dolgushev}.
Here the $\Lie_\infty$ action of $\Tpoly(\R^n)[1]$ on $C_\bullet(\R^n)$ is obtained by pulling back the action of $\Dpoly(\R^n)[1]$ on $C_\bullet(\R^n)$ via $\mU^{\Kontsevich}$. In particular, the  statement that $\mV^{\Shoikhet}$ is a $\Lie_\infty$ morphism of modules implicitly references $\mU^{\Kontsevich}$.

The formality morphisms $\mU^{\Kontsevich}$ and $\mV^{\Shoikhet}$ are given by explicit ``sum of graphs'' formulas:
\begin{align}
\label{equ:mUK}
\mU_k^{\Kontsevich} &= \sum_\Gamma c_\Gamma^{\Kontsevich} D_\Gamma \\
\label{equ:mVS}
\mV_k^{\Shoikhet} &= \sum_{\tilde \Gamma} \tilde c_{\tilde \Gamma}^{\Shoikhet} \tilde D_{\tilde \Gamma}.
\end{align}
Here $\mU_k$ (respectively $\mV_k$) is the $k$-th component of the $\Lie_\infty$ morphism $\mU$ (respectively of $\mV_k$). The top sum runs over the set of isomorphism classes of \emph{Kontsevich graphs} with $k$ type I vertices. For the definition of these graphs we refer the reader to \cite{K1}, an example can be found in Figure \ref{fig:example}. Finally
\[
D_\Gamma : S^k \Tpoly(\R^n)[2] \to \Dpoly(\R^n)[2]
\]
is an operator naturally associated to a Kontsevich graph $\Gamma$.
It implicitly depends on the dimension $n$ of the underlying space $\R^n$.
The coefficients $c_\Gamma^{Kontsevich}$ are numbers.
Similarly, in \eqref{equ:mVS} the sum ranges over all isomorphism classes of \emph{Shoikhet graphs} with $k$ type I vertices (see \cite{shoikhet} for the definition and Figure \ref{fig:example} for an example). The coefficients $\tilde c_{\tilde \Gamma}^{Shoikhet}$ are again numbers and 
\[
\tilde D_{\tilde \Gamma} : S^k \Tpoly[2](\R^n)\otimes  C_\bullet(\R^n) \to \Omega_\bullet(\R^n)
\]
are morphisms naturally associated to Shoikhet graphs, cf. \cite{shoikhet}.

\begin{figure}
\centering
\begin{tikzpicture}
[every edge/.style={draw, -triangle 45}, scale=.5
]
\draw (-6,-4) -- (5,-4);
\node [int] (v2) at (-3,-1) {};
\node [int] (v1) at (-1,2) {};
\node [int] (v5) at (0,-1) {};
\node [int] (v3) at (-3,-4) {};
\node [int] (v4) at (0,-4) {};
\node [int] (v7) at (3,-4) {};
\node [int] (v6) at (2,1) {};
\draw  (v1) edge (v2);
\draw  (v2) edge (v3);
\draw  (v2) edge (v4);
\draw  (v5) edge (v4);
\draw  (v2) edge (v5);
\draw  (v1) edge (v6);
\draw  (v6) edge (v5);
\draw  (v6) edge (v7);

\begin{scope}[shift={(12,-1)}]
\draw (0,0) circle (5);
\end{scope}
\node [int] (v13) at (17,-1) {};
\node [int] (v10) at (12,4) {};
\node [int] (v12) at (12,-4) {};
\node [int] (v11) at (15,0) {};
\node [int] (v9) at (9,-1) {};
\node [int] (v8) at (12,-1) {};
\draw  (v8) edge (v9);
\draw  (v8) edge (v10);
\draw  (v8) edge (v11);
\draw  (v11) edge (v12);
\draw  (v11) edge (v13);
\draw  (v11) edge (v10);
\draw  (v9) edge (v12);
\draw  (v9) edge (v10);
\end{tikzpicture}
\caption{\label{fig:example} An example of a Kontsevich graph (left) and a Shoikhet graph (right). }
\end{figure}
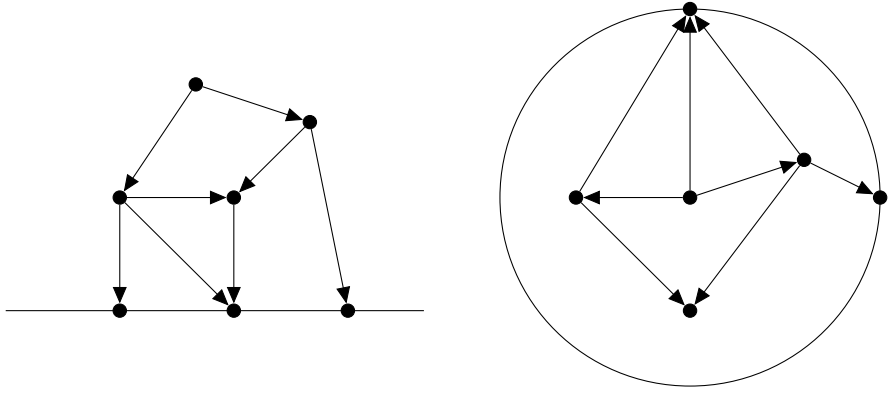

In \cite{vasilystable} formality morphisms given by sum-of-graphs formulas as above were called \emph{stable}.
\begin{defi}[following  \cite{vasilystable}]
\label{def:stable}
A \emph{stable formality morphism on cochains} is a collection of numbers $\{c_\Gamma\}_\Gamma$, one for each Kontsevich graph, such that the formulas 
\begin{align}
\label{equ:mU}
\mU_k &= \sum_\Gamma c_\Gamma D_\Gamma 
\end{align}
define a $\Lie_\infty$ quasi-isomorphism of Lie algebras
\begin{align*}
\mU &\colon \Tpoly(\R^n)[1] \to \Dpoly(\R^n)[1]
\end{align*}
for all $n$($=\dim \R^n$), and such that $\mU_1$ is the Hochschild-Kostant-Rosenberg morphism. 

A \emph{stable formality morphism on cochains and chains} is a stable formality morphism on cochains together with a collection of numbers $\{ \tilde c_{\tilde \Gamma}\}_{\tilde \Gamma}$, one for each Shoikhet graph, such that the formulas 
\begin{align}
\label{equ:mV}
\mV_k &= \sum_{\tilde \Gamma} \tilde c_{\tilde \Gamma} \tilde D_{\tilde \Gamma}
\end{align}
define a $\Lie_\infty$ quasi-isomorphism of modules
\begin{align*}
\mV &\colon C_\bullet(\R^n) \to \Omega_\bullet(\R^n)
\end{align*}
for all $n$, and such that $\mV_0$ is the Connes-Hochschild-Kostant-Rosenberg morphism. 
\end{defi}


\begin{ex}
\label{ex:drinfeld}
In particular, to every Drinfeld associator $\Phi$ one may associate a stable formality morphism of cochains as follows:
\begin{enumerate}
\item To the Alekseev-Torossian Drinfeld associator $\Phi_{AT}$, see \cite{ATassoc, pavol}, we associate the Kontsevich stable formality morphism $\mU^\Kontsevich$.
\item Let $\Phi$ be any Drinfeld associator. 
The Grothendieck-Teichm\"uller group acts freely transitively on the set of Drinfeld associators.
Hence there is a unique element $g$ of the Grothendieck-Teichm\"uller group that maps $\Phi_{AT}$ to $\Phi$. Using the pro-unipotence of the Grothendieck-Teichm\"uller group we may write 
\[
g=\exp(\psi)
\]
For a unique $\psi$ in the  Grothendieck-Teichm\"uller Lie algebra $\grt_1$.
This element $\psi$ may be associated a graph cohomology class in M. Kontsevich's graph complex $\GC_2$ (see \cite[section 3]{grt}), which is represented, say, by some degree 0 cocycle $\gamma\in \GC_2$.
Now $\GC_2$ naturally acts on the set of stable formality morphisms of cochains (see  \cite{vasilystable, brformality, grt}). We define the stable formality morphism associated to $\Phi$ as 
\[
\exp(\gamma) \mU^\Kontsevich.
\]
Note that this is well defined only up to homotopy, since one had to make a choice in picking one representative of the graph cohomology class canonically associated to $\psi$.

In a similar way one may also obtain a stable formality morphism of cochains and chains as discussed in \cite{grtonforms}.
\end{enumerate} 
\end{ex}

\begin{rem}
Definition \ref{def:stable} differs slightly from the one given in \cite[Definition 5.1]{vasilystable} by V. Dolgushev. There, a stable formality morphism was defined as an operad map from a colored operad governing open closed homotopy algebras to a colored operad $\KGra$, satisfying some conditions. Elements of $\KGra$ are essentially linear combinations of Kontsevich graphs. We leave it to the reader to check that both definitions agree. 
\end{rem}

\begin{rem}
Note that all formality morphism constructed as in example \ref{ex:drinfeld} can be globalized, i.~e., they satisfy suitable properties P1)-P5) stated by M. Kontsevich in \cite{K1}.
\end{rem}

\subsection{A remark on signs and prefactors}
The explicit definition of M. Kontsevich's formality morphism, correct with signs and prefactors, and the definition of the symbols $D_\Gamma$ is quite lengthy to state. In fact, a separate paper \cite{AMM} has been written just about the signs and prefactors. It involves conventional choices at various places in the construction. We want to avoid flooding this paper with pages of definitions to fix the signs. 
To still obtain well-defined numbers $c_\Gamma$ we adopt the following conventions:

\begin{enumerate}
\item For each isomorphism class of Kontsevich (resp. Shoikhet) graphs we fix once and for all a representative graph, together with an ordering of the edges. Below, when we introduce certain such graphs, we will indicate the ordering of the edges by writing numbers next to the edges.

\item Our conventions regarding $D_\Gamma$  are assumed to be chosen such that the formulas \eqref{equ:mUK} are correct, for $c_\Gamma$  given by Kontsevich's integral
\begin{align*}
c_\Gamma^{\Kontsevich} &= \int \prod_{(i,j)} \frac{1}{2\pi} d\arg \left(\frac{z_i-z_j}{\bar z_i-z_j} \right)
\end{align*}
where the product is over all edges, in the order that was specified once and for all for this isomorphism class of Kontsevich graphs.
Similarly, we choose our conventions regarding $\tilde D_{\tilde \Gamma}$ such that \eqref{equ:mVS} is correct for $\tilde c_{\tilde \Gamma}$ being the usual Shoikhet integral, without any additional prefactors.
\end{enumerate}

A careful discussion of signs for the Kontsevich morphism, which is somewhat shorter than \cite{AMM} (but still spans many pages) has been given by the author in \cite{mecyccochains}.

\subsection{Homotopies and homotopy invariant functions}
\label{sec:homotopies}
Recall that an $L_\infty$ structure on $\alg{g}$ is a degree 1, square zero coderivation on $S^+\alg g[1]$, the cofree cocommutative coalgebra (without counit) cogenerated by $\alg{g}[1]$. An $L_\infty$ morphism between $L_\infty$ algebras $\alg{g}$ and $\alg{h}$ is a map of coalgebras
\[
f: S^+g[1] \to S^+\alg{h}[1]
\]
compatible with the given coderivations.
Let us say that two $L_\infty$ morphism $f$, $g$ from $\alg g$ to $\alg h$ are directly homotopic if there is an $L_\infty$ morphism 
\[
F \colon \alg g \to \alg h[t,dt]
\]
such that the restriction to $t=0$ (respectively $t=1$) agrees with $f$ (respectively with $g$).
Concretely, $F$ may be written as 
\[
F = f_t + h_t dt
\]
where $f_t$ is a (polynomial) family of $L_\infty$ morphisms interpolating between $f_0=f$ and $f_1=g$. We call the other component, $h_t$ the \emph{homotopy}.

We say that two $L_\infty$ morphisms $f,g$ are \emph{homotopic}, if there is some (finite) tuple of $L_\infty$ morphisms $(a_1,\dots, a_k)$ such that $f$ is directly homotopic to $a_1$, each $a_j$ is directly homotopic to $a_{j+1}$ and $a_k$ is directly homotopic to $g$. Clearly being homotopic is an equivalence relation on the set of $L_\infty$ morphisms from $\alg g$ to $\alg h$.
A function from the set of $L_\infty$ morphisms from $\alg g$ to $\alg h$ to some other set is \emph{homotopy invariant} if it is constant on equivalence classes.  For a more detailed discussion of homotopies between homotopy morphisms we refer the reader to \cite{DotsenkoPoncin}.

The above notion of homotopy may be transferred to stable formality morphisms with minor changes \cite[section 5]{vasilystable}. So let $\mU$, $\mU'$ be stable formality morphisms (say of cochains, the case for cochains and chains is analogous). We say that $\mU$, $\mU'$ are directly homotopic if there is a collection of polynomials $c_\Gamma(t, dt)\in \R[t, dt]$ such that:
\begin{enumerate}
\item The formulas
\[
\tilde \mU_k := \sum_\Gamma c_\Gamma(t, dt) D_\Gamma 
\]
define an $L_\infty$ morphism $\Tpoly(\R^n)\to \Dpoly(\R^n)$ for each $n$.
\item Restricting $\tilde \mU$ to fixed $t$ yields a family of stable formality morphisms interpolating between $\mU$ (for $t=0$) and $\mU'$ (reached at $t=1$).
\end{enumerate}
As above one may split 
\begin{equation}
\label{equ:homotopy}
\tilde \mU = \tilde \mU_t + h_t dt
\end{equation}
where $\tilde \mU_t$ is the restriction of $\tilde \mU$ to fixed $t$ and we call $h_t$ the homotopy.

Again we define the equivalence relation of being homotopic as the transitive closure of the relation of being directly homotopic.
For more details we refer the reader to \cite{vasilystable}.

A function on the set of stable formality morphisms is called homotopy invariant if it is constant on equivalence classes of the above equivalence relation.
Of course this is equivalent to saying that the function takes the same values on directly homotopic stable formality morphisms.

\subsection{Characteristic functions }
We will consider the following four characteristic functions:
\begin{itemize}
\item Let $\mU$ be a stable formality morphism of cochains. We set $f^{\Duflo}(x)=\sum_{j\geq 2} \lambda_j^{\Duflo} x^j$ where $\lambda_j^{\Duflo}=\frac{1}{j} c_{\Gamma^{(I)}_j}-\frac{1}{j} c_{\Gamma^{(II)}_j}$ and $c_{\Gamma^{(I)}_j}$ and $c_{\Gamma^{(II)}_j}$ are the coefficients of the graphs 
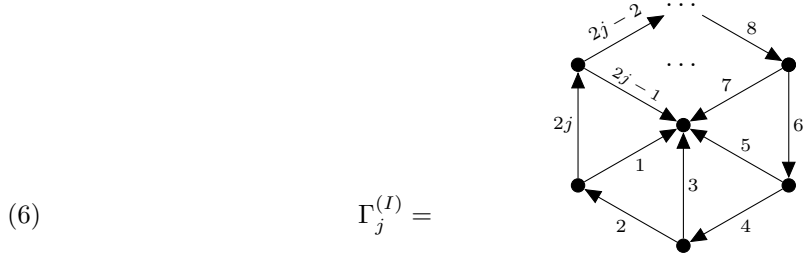
\begin{equation}
\label{equ:wgraph1}
\Gamma^{(I)}_j =
\begin{tikzpicture}
[every edge/.style={draw, triangle 45-},
scale=.8, baseline=-0.65ex, yshift=1.5cm
]
\node [int] (v7) at (0,0) {};
\node [int] at (30:2) {};
\node [int] (v4) at (-30:2) {};
\node [int] (v3) at (-90:2) {};
\node  (v6) at (90:2) {$\dots$};
\node  at (90:1) {$\dots$};
\node [int] (v1) at (150:2) {};
\node [int] (v5) at (30:2) {};
\node [int] (v2) at (210:2) {};
\draw[font=\scriptsize]  (v1) edge node[descr,auto, swap] {$2j$} (v2);
\draw[font=\scriptsize]  (v2) edge node[descr,auto,swap] {2} (v3);
\draw[font=\scriptsize]  (v3) edge node[descr,auto,swap] {4} (v4);
\draw[font=\scriptsize]  (v4) edge node[descr,auto,swap] {6} (v5);
\draw[font=\scriptsize]  (v5) edge node[descr,auto,swap] {8} (v6);
\draw[font=\scriptsize]  (v6) edge node[above, sloped] {$2j-2$} (v1);
\draw[font=\scriptsize]  (v7) edge node[descr, auto] {1} (v2);
\draw[font=\scriptsize]  (v7) edge node[descr,auto] {3} (v3);
\draw[font=\scriptsize]  (v7) edge node[descr,auto] {5} (v4);
\draw[font=\scriptsize]  (v7) edge node[descr,auto] {7} (v5);
\draw[font=\scriptsize]  (v7) edge node[descr, above, sloped, font=\tiny] {$2j-1$} (v1);
\draw (-4,-4) -- (3,-4);
\end{tikzpicture}
\end{equation}
in $\mU_{j+1}$ and 
\begin{equation}
\label{equ:wgraph2}
\Gamma^{(II)}_j=
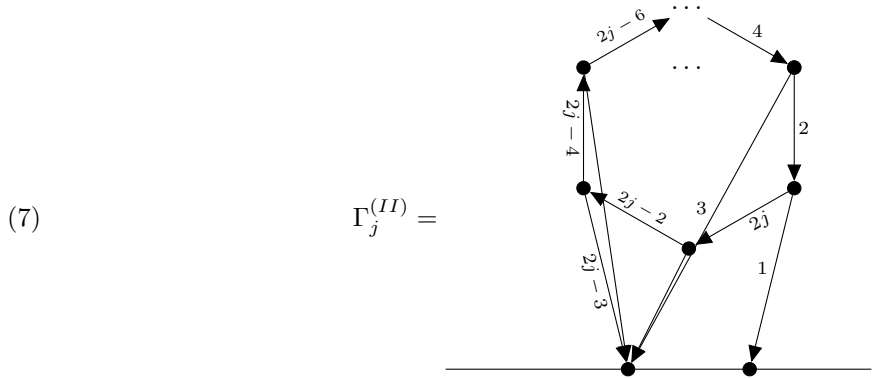
\begin{tikzpicture}
[every edge/.style={draw, triangle 45-}, scale=.8, baseline=-0.65ex,yshift=1.5cm
]
\node [int] (v7) at (-1,-4) {};
\node [int] at (30:2) {};
\node [int] (v4) at (-30:2) {};
\node [int] (v3) at (-90:2) {};
\node  (v6) at (90:2) {$\dots$};
\node  at (90:1) {$\dots$};
\node [int] (v1) at (150:2) {};
\node [int] (v5) at (30:2) {};
\node [int] (v2) at (210:2) {};
\node[int] (vb) at (1,-4) {};
\draw[font=\tiny]  (v1) edge node[descr,below,sloped] {$2j-4$} (v2);
\draw[font=\tiny]  (v2) edge node[descr,above,sloped,swap, font=\tiny] {$2j-2$} (v3);
\draw[font=\tiny]  (v3) edge node[descr,auto,sloped,swap] {$2j$} (v4);
\draw[font=\scriptsize]  (v4) edge node[descr,auto,swap] {2} (v5);
\draw[font=\scriptsize]  (v5) edge node[descr,auto,swap] {4} (v6);
\draw[font=\tiny]  (v6) edge node[above, sloped] {$2j-6$} (v1);
\draw[font=\tiny]  (v7) edge node[descr,sloped, below] {$2j-3$} (v2);
\draw[font=\tiny]  (v7) edge node[descr,sloped,above,font=\tiny] {} (v3);
\draw[font=\scriptsize]  (vb) edge node[descr,auto] {1} (v4);
\draw[font=\scriptsize]  (v7) edge node[descr,auto] {3} (v5);
\draw[font=\scriptsize]  (v7) edge node[descr, above, sloped, font=\tiny] {} (v1);
\draw (-4,-4) -- (3,-4);
\end{tikzpicture}
\end{equation}
in $\mU_j$. The function $f^{\Duflo}$ appears in the proof of Duflo's Theorem through deformation quantization as in \cite{K1}, section 8.\footnote{More precisely, the series  $f^{\Duflo}$ appears in the morphism $I_{strange}$ of \cite[section 8.3.4]{K1}.}
\item Let $\mU$ be a stable formality morphism of cochains. We set $f^{curv}(x)=\sum_{j\geq 2} \lambda_j^{curv} x^j$ where $\lambda_j^{curv}=\frac{1}{j}c_{\Gamma^{(III)}_j}$ and $c_{\Gamma^{(III)}_j}$ is the coefficient of the wheel graph with spokes pointing outwards 
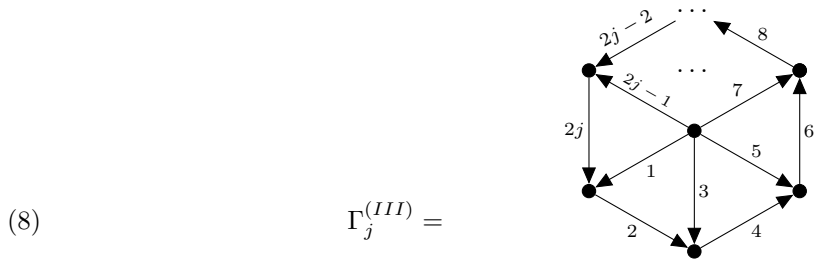
\begin{equation}
\label{equ:spokesout}
\Gamma^{(III)}_j =
\begin{tikzpicture}
[every edge/.style={draw, -triangle 45}, scale=.8, baseline=-0.65ex, yshift=1.5cm,
]
\draw (-4,-4) -- (3,-4);
\node [int] (v7) at (0,0) {};
\node [int] at (30:2) {};
\node [int] (v4) at (-30:2) {};
\node [int] (v3) at (-90:2) {};
\node  (v6) at (90:2) {$\dots$};
\node  at (90:1) {$\dots$};
\node [int] (v1) at (150:2) {};
\node [int] (v5) at (30:2) {};
\node [int] (v2) at (210:2) {};
\draw[font=\scriptsize]  (v1) edge node[descr,auto, swap] {$2j$} (v2);
\draw[font=\scriptsize]  (v2) edge node[descr,auto,swap] {2} (v3);
\draw[font=\scriptsize]  (v3) edge node[descr,auto,swap] {4} (v4);
\draw[font=\scriptsize]  (v4) edge node[descr,auto,swap] {6} (v5);
\draw[font=\scriptsize]  (v5) edge node[descr,auto,swap] {8} (v6);
\draw[font=\scriptsize]  (v6) edge node[above, sloped] {$2j-2$} (v1);
\draw[font=\scriptsize]  (v7) edge node[descr, auto] {1} (v2);
\draw[font=\scriptsize]  (v7) edge node[descr,auto] {3} (v3);
\draw[font=\scriptsize]  (v7) edge node[descr,auto] {5} (v4);
\draw[font=\scriptsize]  (v7) edge node[descr,auto] {7} (v5);
\draw[font=\scriptsize]  (v7) edge node[descr, above, sloped, font=\tiny] {$2j-1$} (v1);
\end{tikzpicture}
\end{equation}
in $\mU_{j+1}$. These graphs appear in \cite{calaquevdb}, \cite{vdb}, \cite{twmodules}, and in particular as a curvature term in the formality morphism with branes \cite{cfrelative, calaquebranes}.

\item Let $(\mU, \mV)$ be a stable formality morphism of cochains and chains. We set $f^{chain}(x)=\sum_{j\geq 2} \lambda_j^{chain} x^j$ where $\lambda_j^{chain}=\frac{1}{j} \tilde c_{\tilde \Gamma_j}$ and $\tilde c_{\tilde \Gamma_j}$ is the coefficient of the graph 
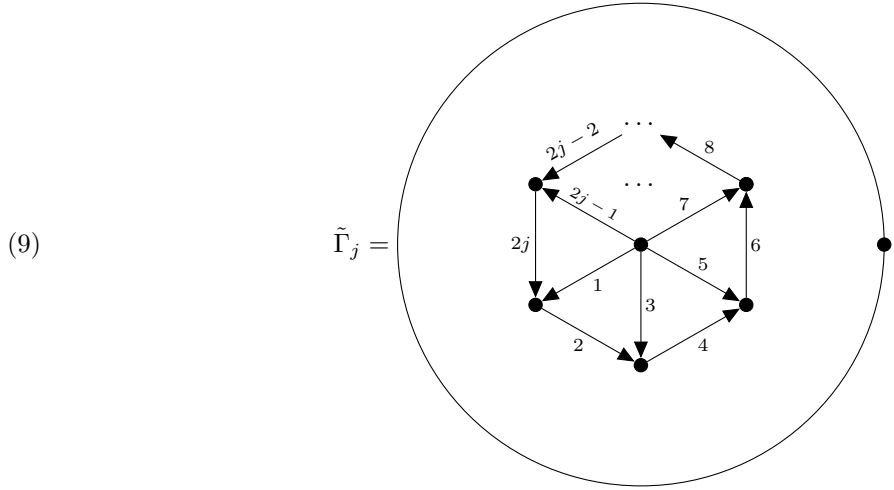
\begin{equation}
\label{equ:chainwheel}
\tilde \Gamma_j=
\begin{tikzpicture}
[every edge/.style={draw, -triangle 45}, scale=.8, baseline=-0.65ex
]
\draw (0,0) circle (4);
\node [int] at (0:4) {};
\node [int] (v7) at (0,0) {};
\node [int] at (30:2) {};
\node [int] (v4) at (-30:2) {};
\node [int] (v3) at (-90:2) {};
\node  (v6) at (90:2) {$\dots$};
\node  at (90:1) {$\dots$};
\node [int] (v1) at (150:2) {};
\node [int] (v5) at (30:2) {};
\node [int] (v2) at (210:2) {};
\draw[font=\scriptsize]  (v1) edge node[descr,auto, swap] {$2j$} (v2);
\draw[font=\scriptsize]  (v2) edge node[descr,auto,swap] {2} (v3);
\draw[font=\scriptsize]  (v3) edge node[descr,auto,swap] {4} (v4);
\draw[font=\scriptsize]  (v4) edge node[descr,auto,swap] {6} (v5);
\draw[font=\scriptsize]  (v5) edge node[descr,auto,swap] {8} (v6);
\draw[font=\scriptsize]  (v6) edge node[above, sloped] {$2j-2$} (v1);
\draw[font=\scriptsize]  (v7) edge node[descr, auto] {1} (v2);
\draw[font=\scriptsize]  (v7) edge node[descr,auto] {3} (v3);
\draw[font=\scriptsize]  (v7) edge node[descr,auto] {5} (v4);
\draw[font=\scriptsize]  (v7) edge node[descr,auto] {7} (v5);
\draw[font=\scriptsize]  (v7) edge node[descr, above, sloped, font=\tiny] {$2j-1$} (v1);
\end{tikzpicture}
\end{equation}
in $\mV_{j+1}$. These graphs determine the character map in deformation quantization, see \cite{CFW}.

\item For any Drinfeld associator $\Phi(X,Y)$ one defines the formal function $f^{\assoc}(x)=\sum_{j\geq 2} \lambda_j^{\assoc} x^j$, where $\lambda_j^{\assoc}$ is the coefficient of $X^{j-1}Y$ in $\Phi(X,Y)$, divided by $j$. So 
\[
\Phi(X,Y)= 1 + \sum_{j\geq 2} j \lambda_j^{\assoc} X^{j-1}Y + (\text{other terms})
\]
 The exponential of the function $-f^{\assoc}$ has been called Duflo function in \cite{AT} and the the $\Gamma$ function in \cite{enriquezgamma}.
\end{itemize}

\begin{ex}
\label{ex:samplecalc}
Several of these characteristic functions have been computed in the literature:
\begin{itemize}
\item Kontsevich computed \cite{K1} that for his stable formality morphism $\mU^\Kontsevich$
\[
f^{\Duflo}=
-\sum_{k=1}^\infty \frac{1}{2k} \frac{B_{2k}}{2(2k)!}x^{2k}
=
-\frac{1}{2} \log\left( \frac{e^{x/2}-e^{-x/2}}{x} \right)
\]
where $B_j$ is the $j$-th Bernoulli number.
In fact, it was shown by B. Shoikhet \cite{shwheel} that $c_{\Gamma^{(I)}_j}=0$ in this case. 

\item For the Kontsevich stable formality morphism $\mU^\Kontsevich$ it has been computed in \cite{vdb} that 
\[
f^{\curv}(x)=
-\sum_{k=1}^\infty \frac{1}{2k} \frac{B_{2k}}{2(2k)!}x^{2k}
=
-\frac{1}{2} \log\left( \frac{e^{x/2}-e^{-x/2}}{x} \right).
\]

Consider also the stable formality morphism of chains and cochains $(\mU^\Kontsevich, \mV^\Shoikhet)$. In this case the integral expressions defining $c_{\Gamma^{(III)}_j}$ and $\tilde c_{\tilde \Gamma_j}$ agree. This also shows that in this case
\[
f^{\chain}(x)=
-\frac{1}{2} \log\left( \frac{e^{x/2}-e^{-x/2}}{x} \right).
\]

\item For the stable formality morphism obtained using the Kontsevich ``$\frac 1 2$-propagator'' (see \cite{ALRT}) it has been shown by S. Merkulov \cite[Appendix A]{merkulovexotic} that
\[
f^{\curv}(x)=
\sum_{k=2}^\infty \frac{\zeta(k)}{k(2\pi i)^k}x^k
=
\log \left( \Gamma\left( 1-\frac{x}{2\pi i}\right) \right) - \frac{\gamma}{2\pi i}x
\]
where $\Gamma$, $\zeta$ and $\gamma$ are the $\Gamma$ function, the Riemann $\zeta$ function and the Euler-Mascheroni constant as usual.

\item It is known (see \cite{lemurakami} or \cite[Example 9.1]{AT}) that for the Knizhnik-Zamolodchikov associator
\[
f^\assoc(x) = \sum_{k=2}^\infty \frac{\zeta(k)}{k(2\pi i)^k}x^k
=
\log \left( \Gamma(1-\frac{x}{2\pi i} \right) - \frac{\gamma}{2\pi i}x.
\]

\item One can check that the even part of $f^\assoc(x)$ must be the same for all Drinfeld associators. Since the Alekseev-Torossian associator \cite{ATassoc, pavol} is even, we obtain from the previous example that for the Alekseev-Torossian associator
\[
f^\assoc(x) = \sum_{k=1}^\infty \frac{k\zeta(2k)}{(2\pi i)^{2k}}x^{2k}
=
-\frac{1}{2} \log\left( \frac{e^{x/2}-e^{-x/2}}{x} \right).
\]
\end{itemize}

\end{ex}



\begin{lemma}[Homotopy Invariance]
\label{lem:hominvariant}
Let $\mU^1, \mU^2$ be stable formality morphisms of cochains that are homotopic. 
Let $f^{\Duflo}_1$, $f^{curv}_1$ and $f^{\Duflo}_2$, $f^{curv}_2$ be the associated characteristic functions as defined above. Then $f^{\Duflo}_1=f^{\Duflo}_2$ and $f^{curv}_1=f^{curv}_2$.

Let furthermore $(\mU^1, \mV^1)$ and, $(\mU^2, \mV^2)$ be homotopic stable formality morphisms of cochains and chains and let $f^{chain}_1$ and  $f^{chain}_2$ be the characteristic functions associated to $\mV^1$ and $\mV^2$ as above. Then $f^{chain}_1=f^{chain}_2$.
\end{lemma}

\begin{proof}[Proof sketch]
It is sufficient to consider only directly homotopic stable formality morphisms (see section \ref{sec:homotopies}). Let us use the notation from equation \eqref{equ:homotopy}. The $dt$-components of the $\Lie_\infty$ relations for $\tilde \mU$ say that 
\begin{equation}
\label{equ:homproof}
\frac{d}{dt} \tilde \mU_t = \pm d_S h_t \pm d_H h_t \pm \co{\mU_t}{h_t}
\end{equation}
where $d_S$ is a term containing the Schouten-Nijenhuis bracket, $d_H$ is (induced from) the Hochschild differential and the bracket is (induced from) the Gerstenhaber bracket. 

To see the invariance for $f^{\curv}$ one notes that (for large enough $n$) the right hand side cannot contain any terms associated to graphs \eqref{equ:spokesout}, as they could be produced by neither the differential $d_S$ and $d_H$, nor by the Gerstenhaber bracket. Hence $f^{\curv}$ must be the same for each $\tilde \mU_t$. For $f^{chain}$ the argument is analogous.

The case of $f^{\Duflo}$ is more difficult, as the right hand side of \eqref{equ:homproof} may contain graphs of the forms \eqref{equ:wgraph1} and \eqref{equ:wgraph2}. Concretely, both can be produced by terms corresponding to a unique graph in $h_t$, namely the following:

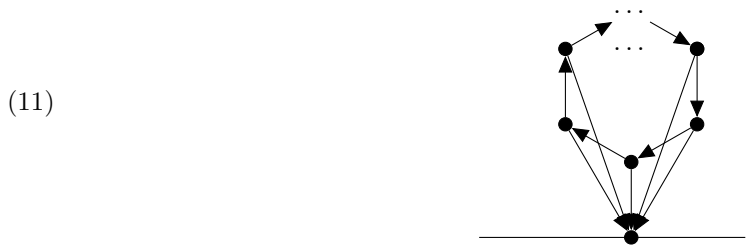
\begin{equation}
\label{equ:whlonebase}
\begin{tikzpicture}
[every edge/.style={draw, triangle 45-}, scale=.5,
baseline=-2ex
]
\node [int] (v7) at (0,-4) {};
\node [int] at (30:2) {};
\node [int] (v4) at (-30:2) {};
\node [int] (v3) at (-90:2) {};
\node  (v6) at (90:2) {$\dots$};
\node  at (90:1) {$\dots$};
\node [int] (v1) at (150:2) {};
\node [int] (v5) at (30:2) {};
\node [int] (v2) at (210:2) {};

\draw  (v1) edge (v2);
\draw  (v2) edge (v3);
\draw  (v3) edge (v4);
\draw  (v4) edge (v5);
\draw  (v5) edge (v6);
\draw  (v6) edge (v1);
\draw  (v7) edge (v2);
\draw  (v7) edge (v3);
\draw  (v7) edge (v4);
\draw  (v7) edge (v5);
\draw  (v7) edge (v1);
\draw (-4,-4) -- (3,-4);
\end{tikzpicture}
\end{equation}
The term $d_H h_t$ (may) contain terms corresponding to the graph \eqref{equ:wgraph2} and the term $\co{\mU_t}{h_t}$ (may) contain terms corresponding to the graph \eqref{equ:wgraph1}.

However, computing the signs and prefactors both contributions are equal and hence $f^{\Duflo}$ remains unchanged.

Note also that graphs of the form 
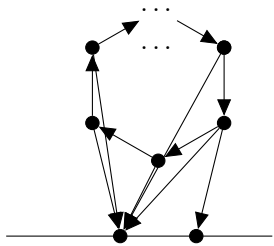
\begin{equation*}
\begin{tikzpicture}
[every edge/.style={draw, triangle 45-}, scale=.5, baseline=-0.65ex,yshift=1.5cm
]
\node [int] (v7) at (-1,-4) {};
\node [int] at (30:2) {};
\node [int] (v4) at (-30:2) {};
\node [int] (v3) at (-90:2) {};
\node  (v6) at (90:2) {$\dots$};
\node  at (90:1) {$\dots$};
\node [int] (v1) at (150:2) {};
\node [int] (v5) at (30:2) {};
\node [int] (v2) at (210:2) {};
\node[int] (vb) at (1,-4) {};
\draw[font=\tiny]  (v1) edge node[descr,below,sloped] {} (v2);
\draw[font=\tiny]  (v2) edge node[descr,above,sloped,swap, font=\tiny] {} (v3);
\draw[font=\tiny]  (v3) edge node[descr,auto,sloped,swap] {} (v4);
\draw[font=\scriptsize]  (v4) edge node[descr,auto,swap] {} (v5);
\draw[font=\scriptsize]  (v5) edge node[descr,auto,swap] {} (v6);
\draw[font=\tiny]  (v6) edge node[above, sloped] {} (v1);
\draw[font=\tiny]  (v7) edge node[descr,sloped, below] {} (v2);
\draw[font=\tiny]  (v7) edge node[descr,sloped,above,font=\tiny] {} (v3);
\draw[font=\scriptsize]  (vb) edge node[descr,auto] {} (v4);
\draw[font=\scriptsize]  (v7) edge node[descr,auto] {} (v4);
\draw[font=\scriptsize]  (v7) edge node[descr,auto] {} (v5);
\draw[font=\scriptsize]  (v7) edge node[descr, above, sloped, font=\tiny] {} (v1);
\draw (-4,-4) -- (3,-4);
\end{tikzpicture}
\end{equation*}
in $h_t$ do not contribute since the two terms of the form \eqref{equ:wgraph2} that can be produced through $d_S h_t$ occur with opposite signs and hence cancel.
\end{proof}


\subsection{Main result}

The main result of this paper is the following:

\begin{thm}[{Partially contained in \cite{K1}, \cite[section 10]{grt}}]
\label{thm:fs}

\hfill

\begin{enumerate}
\item 
Let $\mU$ be a stable formality morphism of cochains. Then 
\[
f^{Duflo}=f^{curv}.
\]

\item If $\mU$ is obtained from a Drinfeld associator $\Phi$ according to the procedure of Example \ref{ex:drinfeld}, then furthermore
\[
f^{\Duflo}=f^{\curv}=f^{\assoc}.
\]

\item Let $(\mU, \mV)$ be an extension of $\mU$ to a stable formality morphism of cochains and chains. Then 
\[
f^{\Duflo}=f^{\curv}=f^{\chain}.
\]
\end{enumerate}
\end{thm}
The above Theorem can in fact almost be extracted from existing literature. The fact that $f^{\Duflo} = f^{curv}$ is essentially contained in some form in \cite{K1}, and the fact that $f^{curv}=f^{assoc}$ is contained (in an albeit sketchy way) in \cite{grt}. Nevertheless we will give a self-contained proof in section \ref{sec:theproof} below.

\begin{rem}
 In fact, the even part of the characteristic functions above is the same for all stable formality morphisms and agrees with the function
\[
-\frac 1 2 \log \frac{e^{x/2}-e^{-x/2}}{x} = -\sum_{j\geq 1}\frac{B_{2j}}{4j(2j)!}x^{2j}
\]
\end{rem}



%
%
%


%

\subsection*{Acknowledgements}
The author is very grateful for many discussions with Vasily Dolgushev.
I thank the Swiss National Science Foundation (grants PDAMP2\_137151 and 200021\_150012) for partial support.
Part of this work has been written while the author was a Junior Fellow of the Harvard Society of Fellows.

\section{Action of the graph complex}
\label{sec:action}
M. Kontsevich's graph complex $\GC_2$ is a complex formed by formal series of (isomorphism classes of) undirected, at least trivalent, connected graphs.
The simplest non-trivial example of a graph giving rise to an element of $\GC_2$ is the tetrahedron graph
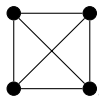
\[
\begin{tikzpicture}
\node[int] (v1) at (0,0) {};
\node[int] (v2) at (1,0) {};
\node[int] (v3) at (0,1) {};
\node[int] (v4) at (1,1) {};
\draw (v1)--(v2)--(v3)--(v4)--(v1)--(v3) (v2)--(v4);
\end{tikzpicture}.
\]
For more details, and the (lengthy) definition of $\GC_2$ we refer the reader to \cite[section 3]{grt}, \cite[section 6]{vasilystable}.
For us, the important fact is that there is a map of dg Lie algebras from $\GC_2$  to the Chevalley complex of $\Tpoly(\R^n)$ for each $n$. In particular, closed degree zero elements of $GC_2$ give rise to $\Lie_\infty$-derivations of $\Tpoly(\R^n)[1]$. Denote the space of closed degree 0 elements by $GC_{2,cl}^0\subset GC_2$. It is a pro-nilpotent Lie algebra, and is the Lie algebra of a prounipotent group 
\[
\ExpGC
\] 
which may be realized as the grouplike elements in the completed universal enveloping algebra of $GC_{2,cl}^0$. The action of $GC_{2,cl}^0$ on $\Tpoly(\R^n)[1]$ by $\Lie_\infty$-derivations integrates to an action of $\ExpGC$ on $\Tpoly(\R^n)[1]$ by $\Lie_\infty$-automorphisms.
 It is then not hard to check that precomposition yields an action of $\ExpGC$ on the set of stable formality morphisms (of cochains). It is clear that this action descends to an action of the homotopy classes of stable formality morphisms.
V. Dolgushev showed the following Theorem, which is important for us.
\begin{thm}[\cite{vasilystable}]
\label{thm:transitive}
The induced action of $\ExpGC$ on the set of homotopy classes of stable formality morphisms is transitive.
\end{thm} 

\section{Proof of the Theorem \ref{thm:fs}}
\label{sec:theproof}
First, let us reduce the statement to the cases involving only stable formality morphisms of cochains by showing that $f^{curv}=f^{chain}$.
For this, consider the coefficient of the graph 
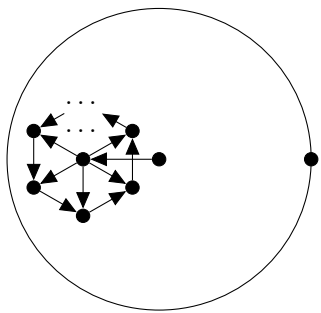
\[
\begin{tikzpicture}
[every edge/.style={draw, -triangle 45}, scale=.5
]
\draw (2,0) circle (4);
\node [int] at (0:6) {};
\node [int] (v7) at (0,0) {};
\node [int] (v4) at (-30:1.5) {};
\node [int] (v3) at (-90:1.5) {};
\node  (v6) at (90:1.5) {$\dots$};
\node  at (90:.75) {$\dots$};
\node [int] (v1) at (150:1.5) {};
\node [int] (v5) at (30:1.5) {};
\node [int] (v2) at (210:1.5) {};
\node [int] (vc) at (2,0) {};
\draw  (v1) edge (v2);
\draw  (v2) edge (v3);
\draw  (v3) edge (v4);
\draw  (v4) edge (v5);
\draw  (v5) edge (v6);
\draw  (v6) edge (v1);
\draw  (v7) edge (v2);
\draw  (v7) edge (v3);
\draw  (v7) edge (v4);
\draw  (v7) edge (v5);
\draw  (v7) edge (v1);
\draw  (vc) edge (v7);
\end{tikzpicture}
\]
in the $\Lie_\infty$ relation for modules. Terms can be contributed by the graph $\Gamma^{(III)}_j$ (see \eqref{equ:spokesout})
and by the graph
$\tilde \Gamma_j$ (see \eqref{equ:chainwheel})
and by no other graphs. Checking the prefactors, It follows that the coefficients need to be equal, up to possibly an overall sign, which depends on conventions, but not on the particular stable formality morphism chosen. However, for the Kontsevich/Shoikhet morphism our conventions and example \ref{ex:samplecalc} say that the sign is ``+'', hence it must be ``+'' for any stable formality morphism.

Next let us turn to the statement that $f^{\curv}=f^{\Duflo}$.
By Lemma \ref{lem:hominvariant} and Theorem \ref{thm:transitive} it suffices to show the following two statements.

\begin{enumerate}
 \item \label{item1} For one particular stable formality morphism $f^{\curv}=f^{\Duflo}$.
 \item The action of degree zero cocycles in $\GC_2$ leaves invariant the expression $f^{\Duflo}(x)-f^{curv}(x)$.
\end{enumerate}

We take for the particular formality morphism that constructed by M. Kontsevich, i. e. $\mU^\Kontsevich$. In this case item \ref{item1} above is settled by example \ref{ex:samplecalc}.

%


Next consider the action of a degree zero cocycle $\Gamma\in \GC_2$. By the explicit description of the action it cannot change the coefficient of the graph $\Gamma^{(II)}_j$ (see \eqref{equ:wgraph2}) in a stable formality morphism. Furthermore it changes both the coefficients of the graphs $\Gamma^{(II)}_j$ and $\Gamma^{(III)}_j$ (see \eqref{equ:wgraph1}, \eqref{equ:spokesout}) by the coefficient of the wheel graph 
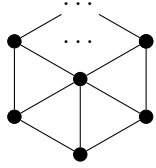
\[
\begin{tikzpicture}[scale=.5]
\node [int] (v7) at (0,0) {};
\node [int] at (30:2) {};
\node [int] (v4) at (-30:2) {};
\node [int] (v3) at (-90:2) {};
\node  (v6) at (90:2) {$\dots$};
\node  at (90:1) {$\dots$};
\node [int] (v1) at (150:2) {};
\node [int] (v5) at (30:2) {};
\node [int] (v2) at (210:2) {};
\draw  (v1) edge (v2);
\draw  (v2) edge (v3);
\draw  (v3) edge (v4);
\draw  (v4) edge (v5);
\draw  (v5) edge (v6);
\draw  (v6) edge (v1);
\draw  (v7) edge (v2);
\draw  (v7) edge (v3);
\draw  (v7) edge (v4);
\draw  (v7) edge (v5);
\draw  (v7) edge (v1);
\end{tikzpicture}
\]
in $\Gamma$. In particular the quantity $f^{\Duflo}(x)-f^{\curv}(x)$ is unchanged.
Hence we have shown that $f^{\curv}=f^{\Duflo}$ for all stable formality morphisms.

To show the final assertion of Theorem \ref{thm:fs} the proof is similar and has been given in \cite{grt}. We recall it here. It clearly suffices to show the following.
\begin{enumerate}
\item \label{item_1} For the Kontsevich stable formality morphism and the Alekseev-Torossian Drinfeld associator, $f^{\assoc}=f^{\curv}$.

\item The difference $f^{\assoc}-f^{\curv}$ is invariant under the action of the Grothendieck-Teichm\"uller Lie algebra $\grt$, where to define its action on stable formality morphisms one uses the map from $\grt_1$ to $H(\GC_2)$ as in Example \ref{ex:drinfeld}.
\end{enumerate}

Again, item \ref{item_1} has been settled by Example \ref{ex:samplecalc}.
Furthermore the cycle in graph homology $s_n$ that picks out the coefficient of the wheel graph with $n$ spokes ($n$ odd) is shown in \cite[Proposition 9.1]{grt} to correspond to the cochain of the Grothendieck-Teichm\"uller Lie algebra $\grt_1$ that picks out the coefficient of 
\[
\ad_X^{n-1}Y
\]
of elements in $\grt_1$. The action of some $\grt_1$ element on a Drinfeld associator changes the coefficient of $X^{n-1}Y$ of the associator by precisely this term. Hence Theorem \ref{thm:fs} follows.

\section{Application: Star products on duals of Lie algebras}
Let $\alg g$ be any Lie algebra, $U\alg g$ its universal enveloping algebra, and $S\alg g$ the symmetric algebra. The Poincar\'e-Birkhoff-Witt isomorphism   
\[
\phi_{\PBW} \colon S\alg g \to U\alg g
\]
endows $S\alg g$ with an associative (but not necessarily commutative) product $\star_{\PBW}$ via pullback, i. e.,
\[
p \star_{\PBW} q := \phi_{\PBW}^{-1} (\phi_{\PBW}(p)\phi_{\PBW}(q))
\] 
for any $p,q\in S\alg g$.

Furthermore, for any Lie algebra $\alg g$ the dual space $\alg g^*$ carries a canonical Poisson structure, the Kirillov-Kostant Poisson structure.
 A stable formality morphism provides us (in particular) with an associative product $\star$ on $S\alg g$. This product in general depends on the stable formality morphism chosen. However, it is an elementary exercise to check that any such product is the pull-back of $\star_{\PBW}$ via an automorphism of the vector space $S\alg g$ of the form
\begin{equation}
\label{equ:Phidefi}
\Psi = \exp\left( \sum_{j\geq 2} c_j \tr(\ad_\p^j)  \right)
\end{equation}
for some constants $c_j$. Here 
\[
\tr(\ad_\p^j) := 
f_{i_1}^{k_1i_j}f_{i_2}^{k_2i_1}\cdots f_{i_j}^{k_ji_{j-1}}
\p_{k_1}\cdots \p_{k_j}
\]
are differential operators where $f^{ab}_c$ are the structure constants of the Lie algebra and summation over repeated indices is assumed.
Note that the constants $c_j$ are not characters of the stable formality morphism, i. e., they may change upon changing the stable formality morphism to a gauge equivalent one.
However, there is the following result.
\begin{prop}
Given a stable formality morphism define the formal series $f(x):=-\sum_j \frac{(-1)^j}{j} c_j x^j$, where the $c_j$ are as in \eqref{equ:Phidefi}. If the stable formality morphism is such that the weights $c_{\Gamma_j^{(I)}}$ of graphs $\Gamma_j^{(I)}$ (cf. \eqref{equ:wgraph1}) vanish for all $j$, then 
$f$ agrees with the characteristic function defined above, i.e.,
\[
f=f^{\Duflo}=f^{\curv}.
\]
\end{prop}
\begin{proof}
Under the assumptions given $\lambda_j^{\Duflo}=-\frac 1 j c_{\Gamma^{(II)}_j}$. However, it is not hard to check that $\star_{PBW}$ does not contain terms corresponding to graphs $\Gamma^{(II)}_j$. They have to be produced via pullback with $\tr(\ad_\p^j)$ and hence the respective coefficients need to agree, up to a combinatorial prefactor, independent of the stable formality morphism under consideration. Unwinding conventions left implicit in this paper the combinatorial prefactor could be computed.
However, to settle the prefactors it is also sufficient to check that both characteristic functions agree for one stable formality morphism for which all of the coefficients of $f$ are non-zero. It has been shown by C. Rossi \cite{rossiexplicit} that for the Kontsevich formality morphism with $\frac 1 2$-propagator, $f=f^\curv$ (cf. also Example \ref{ex:samplecalc}). Since in this case all $c_j\neq 0$, the combinatorial prefactors must all be $+1$. 
\end{proof}

In the special case of the Kontsevich stable formality morphism, $\Psi$ becomes the Duflo morphism, hence the name of $f^{\Duflo}$. Special cases of the above proposition have been shown in \cite{K1}, \cite[Appendix F]{FWirrationality} and \cite{calaquerelations, rossiexplicit}. 

\section{The ``two branes'' case}
The above results may be extended slightly to apply to the formality morphisms ``with branes'' introduced by Calaque, Felder, Ferrario and Rossi \cite{calaquebranes}. 
In particular, one may identify a characteristic function for ``stable versions'' (i. e. given by sum-of-graphs formulas) of such morphisms, which has been used implicitly in \cite{calaquerelations, rossiexplicit}.
This function turns out to agree with the characteristic functions discussed above. 

Let us begin by reviewing the results of \cite{calaquebranes}. 
Consider the polynomial (or exterior) algebras $A=\R[X_1,\dots, X_n]$ and $B=\R[\xi_1,\dots, \xi_n]$ where the formal variables $X_1,\dots, X_n$ live in degree 0, while the formal variables $\xi_1,\dots, \xi_n$ live in degree 1. $A$ and $B$ are Koszul dual algebras. One may show this by showing that the Koszul complex 
\[
A\otimes B^*
\]
has cohomology $\R$. Note also that $A\otimes B^*$ carries a natural $A$-$B$ bimodule structure.

 The first result of \cite{calaquebranes} is an explicit construction of an $\Ass_\infty$ $A$-$B$ bimodule structure on $K=\R$. It was shown in \cite{rossikoszul} that the bimodule $K$ is in fact $\Ass_\infty$ quasi-isomorphic to $A\otimes B^*$.
 
One may package $A$, $B$ and $K$ into an $A_\infty$ category $\mathsf{Cat}_\infty(A,B,K)$ (notation as in \cite{calaquebranes}) with objects $A$ and $B$ and the space of morphisms between $A$ and $B$ being $K$. The second result of \cite{calaquebranes} is the construction of a $\Lie_\infty$ morphism 
\[
\Tpoly(\R^n)[1] \to C(\mathsf{Cat}_\infty(A,B,K))[1]
\]
where the right hand side is the Hochschild complex of $\mathsf{Cat}_\infty(A,B,K)$.
This morphism contains the Kontsevich formality morphism $\mU^{\Kontsevich}$ from above.

One may package both the $\Ass_\infty$ bimodule structure and the $\Lie_\infty$ morphism into a ``non-flat'' $\Lie_\infty$ morphism, i.~e., a $\Lie_\infty$ morphism with non-vanishing zeroth term, which encodes the bimodule structure. This morphism is also given by a sum-of-graphs formula of the form
\[
\mW_k^{\CFFR} = \sum_{\bar \Gamma} \bar c_{\bar\Gamma}^{\CFFR} D_{\bar\Gamma}.
\]
Here the graphs summed over are essentially Kontsevich graphs, possibly with one distinguished type II vertex.
For a more precise definition, we refer the reader to \cite{calaquebranes}.

In analogy with definition \ref{def:stable} above we may define a stable formality morphism of Calaque, Felder, Ferrario and Rossi (CFFR) type to be a collection of numbers $c_{\bar\Gamma}$ such that 
\[
\mW_k = \sum_{\bar \Gamma} c_{\bar\Gamma} D_{\bar\Gamma}.
\]
defines a non-flat $\Lie_\infty$ morphism for all $n$, and such that (i) the restriction to Kontsevich type graphs yields a stable formality morphism and (ii) the two graphs below have coefficient 1.
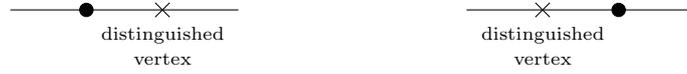
\begin{align*}
\begin{tikzpicture}[every text node part/.style={align=center, font=\scriptsize}, scale=1]
\draw (-1,0)--(2,0);
\node[int] at (0,0) {};
\node [xit, label=-90:{distinguished \\vertex}] (dist) at (1,0) {};
\begin{scope}[xshift=6cm]
\draw (-1,0)--(2,0);
\node[int] at (1,0) {};
\node [xit, label=-90:{distinguished \\vertex}] (dist) at (0,0) {};
\end{scope}
\end{tikzpicture}
\end{align*}
These graphs are the leading contribution to the bimodule structure.

Such stable formality morphisms possess a characteristic function
\[
f^{\brane}(x) = \sum_{j\geq 2}\lambda_j^{\brane} x^j
\]
where $\lambda_j^{\brane} =\frac 1 j c_{\Gamma_j^{I}} + \frac 1 j \bar c_{\bar \Gamma_j}$, with $\Gamma_j^{I}$ as depicted in \eqref{equ:wgraph1}, and $\bar \Gamma_j$ as follows:
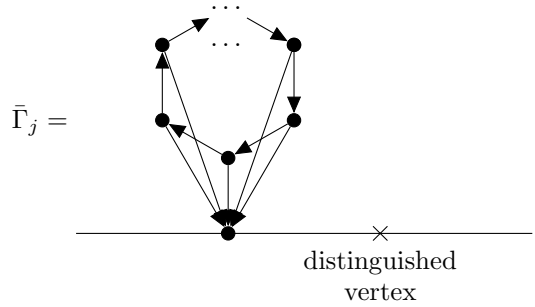
\[
\bar \Gamma_j = 
\begin{tikzpicture}
[every edge/.style={draw, triangle 45-}, scale=.5, baseline=-0.65ex, yshift=1cm,
every text node part/.style={align=center}]
\node [int] (v7) at (0,-4) {};
\node [int] at (30:2) {};
\node [int] (v4) at (-30:2) {};
\node [int] (v3) at (-90:2) {};
\node  (v6) at (90:2) {$\dots$};
\node  at (90:1) {$\dots$};
\node [int] (v1) at (150:2) {};
\node [int] (v5) at (30:2) {};
\node [int] (v2) at (210:2) {};
\node [xit, label=-90:{distinguished \\vertex}] (dist) at (4,-4) {};

\draw  (v1) edge (v2);
\draw  (v2) edge (v3);
\draw  (v3) edge (v4);
\draw  (v4) edge (v5);
\draw  (v5) edge (v6);
\draw  (v6) edge (v1);
\draw  (v7) edge (v2);
\draw  (v7) edge (v3);
\draw  (v7) edge (v4);
\draw  (v7) edge (v5);
\draw  (v7) edge (v1);
\draw (-4,-4) -- (8,-4);
\end{tikzpicture}
\]

It may be verified that $f^{\brane}$ is indeed a characteristic function, i.~e., it does not change when changing the stable formality morphism of CFFR type to a homotopic one. Note that this is not true if one omits the term $c_{\Gamma_j^{I}}$ from the definition.
The characteristic function $f^{\brane}$ is implicitly used in \cite{calaquerelations, rossiexplicit, twmodules}, where it is shown to agree with $f^{\curv}$ for two special stable formality morphisms of CFFR type.
We have the following general result:
\begin{prop}
$f^{\brane}=f^{\curv}$ for all stable formality morphisms of CFFR type.
\end{prop}
\begin{proof}[Proof sketch.]
The statement is equivalent to saying that the coefficients of the terms associated to the graphs 
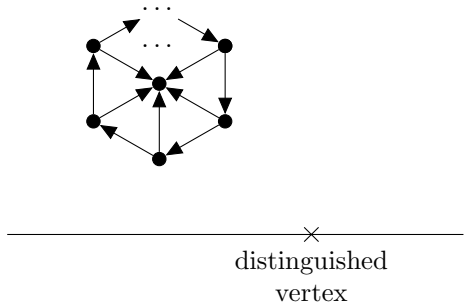
\[
\begin{tikzpicture}
[every edge/.style={draw, triangle 45-},
scale=.5, baseline=-0.65ex, yshift=1.5cm, every text node part/.style={align=center}
]
\node [int] (v7) at (0,0) {};
\node [int] at (30:2) {};
\node [int] (v4) at (-30:2) {};
\node [int] (v3) at (-90:2) {};
\node  (v6) at (90:2) {$\dots$};
\node  at (90:1) {$\dots$};
\node [int] (v1) at (150:2) {};
\node [int] (v5) at (30:2) {};
\node [int] (v2) at (210:2) {};

\node [xit, label=-90:{distinguished \\vertex}] (dist) at (4,-4) {};
\draw  (v1) edge (v2);
\draw  (v2) edge (v3);
\draw  (v3) edge (v4);
\draw  (v4) edge (v5);
\draw  (v5) edge (v6);
\draw  (v6) edge (v1);
\draw  (v7) edge (v2);
\draw  (v7) edge (v3);
\draw  (v7) edge (v4);
\draw  (v7) edge (v5);
\draw  (v7) edge (v1);
\draw (-4,-4) -- (8,-4);
\end{tikzpicture}
\]
in the $\Lie_\infty$ relations vanish (for $n$ big enough). 
\end{proof}

\end{document}